\documentclass[9pt,shortpaper,twoside,web]{IEEEtran}
\usepackage{cite}

\ifCLASSINFOpdf
  \usepackage[pdftex]{graphicx}
  \usepackage{epstopdf}
  \usepackage{subfigure}
  \graphicspath{{./}{./figures/}}
\else
\fi
\usepackage{xcolor}

\usepackage{amsmath}
\usepackage{amssymb}

\def\R{\mathbb{R}}
\def\N{\mathbb{N}}

\newtheorem{nnassumption}{\bf Assumption}

\newtheorem{nntheorem}{\bf Theorem}
\newenvironment{theorem}{\begin{nntheorem}\it}{\end{nntheorem}}
\newtheorem{nncorollary}{\bf Corollary}

\newtheorem{nndefinition}{\bf Definition}

\newtheorem{nnproposition}{\bf Proposition}

\newtheorem{nnproblem}{\bf Problem}

\newtheorem{nnlemma}{\bf Lemma}

\newtheorem{nnremark}{\bf Remark}
\newenvironment{remark}{\begin{nnremark} \rm }{\hfill \hspace*{1pt}\hfill $\circ$\end{nnremark}}
\newtheorem{nnexample}{\bf Example}

\newenvironment{proof}{{\bf Proof.}}{\hfill \hspace*{1pt}\hfill $\Box$}

\allowdisplaybreaks

\begin{document}
%
\title{Local Output Feedback Stabilization of a Nonlinear Kuramoto–Sivashinsky equation}
%
%
%

\author{Hugo~Lhachemi
\thanks{Hugo Lhachemi is with Universit{\'e} Paris-Saclay, CNRS, CentraleSup{\'e}lec, Laboratoire des signaux et syst{\`e}mes, 91190, Gif-sur-Yvette, France (e-mail: hugo.lhachemi@centralesupelec.fr).

This work has been partially supported by ANR PIA funding: ANR-20-IDEES-0002.
}
}

%
%

\markboth{}%
{Lhachemi \MakeLowercase{\textit{et al.}}}
%



\maketitle

\begin{abstract}
This paper is concerned with the local output feedback stabilization of a nonlinear Kuramoto–Sivashinsky equation. The control is located at the boundary of the domain while the measurement is selected as a Neumann trace. This choice of system output requires the study of the system trajectories in $H^2$-norm. Moreover, the choice of the actuation/sensing scheme is discussed and adapted in function of the parameters of the plant in order to avoid the possible loss of controllability/observability property of certain eigenvalues of the underlying operator. This leads in certain cases to a multi-input multi-output control design procedure. The adopted control strategy is finite dimensional and relies on spectral reduction methods. We derive sufficient conditions ensuring the local exponential stabilization of the plant. These control design constraints are shown to be feasible provided the order of the controller is selected to be large enough, ensuring that the reported control design procedure is systematic.
\end{abstract}

\begin{IEEEkeywords}
Kuramoto–Sivashinsky PDE, output feedback stabilization, finite-dimensional control, multi-input multi-output.
\end{IEEEkeywords}

%
\IEEEpeerreviewmaketitle

\section{Introduction}\label{sec: Introduction}

This work is concerned with the problem of local exponential stabilization of a Kuramoto-Sivashinsky equation (KSE). This equation describes the diffusive instabilities in a laminar flame front~\cite{sivashinsky1977nonlinear,nicolaenko1986some} and takes the form of a nonlinear parabolic partial differential equation (PDE). The topic of feedback stabilization of linear parabolic PDEs has been intensively studied in the literature~\cite{boskovic2001boundary,liu2003boundary,prieur2018boundary}. Most control design approaches rely either on backstepping transformations~\cite{krstic2008boundary} or spectral reduction methods~\cite{russell1978controllability,coron2004global,coron2006global}. In particular, spectral reduction methods have been reported for the state-feedback stabilization of the linear KSE by means of boundary control in~\cite{cerpa2010null,cerpa2017control} with extensions to delay boundary control in~\cite{guzman2019stabilization}, output feedback in~\cite{katz2021finite}, and the nonlinear KSE in~\cite{al2018linearized,katz2021regional}. The output feedback stabilization of the nonlinear KSE was considered in~\cite{kang2018distributed} using fully distributed inputs (i.e., in-domain actuators covering the whole spatial domain) and multiple sensors. 

This paper addresses the local output feedback stabilization of a nonlinear KSE by means of spectral reduction methods. The reported approach leverages finite-dimensional control strategies as initiated in the pioneer papers~\cite{curtain1982finite,sakawa1983feedback,balas1988finite,harkort2011finite}. We particularly take advantage of a finite-dimensional observer-based controller architecture originally introduced in \cite{sakawa1983feedback} augmented with a LMI based procedure initially reported in~\cite{katz2020constructive}. This procedure was enhanced and generalized in a systematic manner in~\cite{lhachemi2020finite,lhachemi2021nonlinear} for general reaction-diffusion PDEs with Dirichlet/Neumann/Robin boundary control and Dirichlet/Neumann measurement while performing the control design directly with the control input instead of it time derivative (see~\cite{curtain2012introduction} for an introduction to boundary control systems). This type of approach has been shown to be very efficient for the local output feedback stabilization of linear reaction-diffusion PDEs in the presence of a saturation~\cite{lhachemi2021local}, the global stabilization of linear-reaction-diffusion PDEs in the presence of a Lipchitz continuous sector nonlinearity in the application of the boundary control~\cite{lhachemi2021nonlinear}, as well as the global stabilization of semilinear reaction-diffusion PDE with globally Lipchitz nonlinearity~\cite{lhachemi2021global}. While all these approaches have been applied to reaction-diffusion PDEs, they easily extend to the linear KSE under suitable controllability/observability properties~\cite{katz2021finite}. However, the case of the output feedback stabilization of the nonlinear KSE remains open. This is because the corresponding nonlinearity is not globally Lipchitz continuous and involves terms that are unbounded with respect to the norm of the underlying state-space.

In this context, this paper solves the problem of local output feedback stabilization of the nonlinear KSE. The system output is selected as a Neumann trace. Compared to the studies cited in the previous paragraph which solely dealt with the $H^1$-norm, this system output requires for well-posedness considerations and in the context of this paper the study of the stability of the PDE trajectories in $H^2$-norm. The subsequent sufficient LMI stability conditions are shown to be always feasible when selecting the order of the observer large enough. Finally, it is worth being noted that, depending upon the considered actuation/sensing setting and the value of a certain parameter of the KSE, certain unstable eigenvalues of the plant may be uncontrollable or/and unobservable~\cite{cerpa2010null,cerpa2017control}. In order to circumvent this pitfall, we discuss choices of the actuation and sensing scheme in order to successfully address all the possible configurations for the parameters of the plant. In the most stringent configurations, this requires the introduction of 2 scalar control inputs and 2 scalar outputs, yielding a multi-input multi-output (MIMO) control design problem. This is the first time that the systematic control design procedures reported in~\cite{lhachemi2020finite,lhachemi2021nonlinear} are extended to a MIMO scenario.

The paper is organized as follows. The control design for the nonlinear KSE by means of Dirichlet boundary control and Neumann boundary measurement is addressed in Section~\ref{sec2}. Since this actuation/sensing scheme fails to address a countable number of configurations for the parameters of the plant, we discuss in Section~\ref{sec3} how to circumvent this issue by considering alternative actuation/sensing schemes. A numerical illustration is reported in Section~\ref{sec4}. Finally, concluding remarks are formulated in Section~\ref{sec: conclusion}.

\textbf{Notation.}
Real spaces $\R^n$ are equipped with the usual Euclidean norm denoted by $\Vert\cdot\Vert$. The associated induced norms of matrices are also denoted by $\Vert\cdot\Vert$. For any two vectors $X$ and $Y$, $ \mathrm{col} (X,Y)$ represents the vector $[X^\top,Y^\top]^\top$. The space of square integrable functions on $(0,1)$ is denoted by $L^2(0,1)$ and is endowed with the inner product $\langle f , g \rangle = \int_0^1 f(x) g(x) \,\mathrm{d}x$. The associated norm is denoted by $\Vert \cdot \Vert_{L^2}$. For an integer $m \geq 1$, $H^m(0,1)$ stands for the $m$-order Sobolev space and is endowed with its usual norm $\Vert \cdot \Vert_{H^m}$. For any symmetric matrix $P \in\R^{n \times n}$, $P \succeq 0$ (resp. $P \succ 0$) indicates that $P$ is positive semi-definite (resp. positive definite).

\section{Control design for Kuramoto-Sivashinsky equation with Dirichlet boundary actuation}\label{sec2}

\subsection{Problem description and preliminary spectral reduction}

We consider in this paper the local output feedback stabilization of the KSE described by
\begin{subequations}\label{eq: RD system 1}
\begin{align}
& z_t + z_{xxxx} + \lambda z_{xx} + \mu z z_x = 0 \label{eq: RD system 1 - 1} \\
& z(t,0) = u(t) \label{eq: RD system 1 - 2} \\
& z(t,1) = z_{xx}(t,0) = z_{xx}(t,1) = 0 \label{eq: RD system 1 - 3} \\
& z(0,x) = z_0(x) . \label{eq: RD system 1 - 4}
\end{align}
\end{subequations}
for $t > 0$ and $x \in(0,1)$ where $\lambda > 0$ and $\mu\in\R$. Here we have $z(t,\cdot) \in L^2(0,1)$ the state, $z_0 \in L^2(0,1)$ the initial condition, and $u(t) \in\R$ the boundary control input. The system output is selected as the right Neumann trace:
\begin{equation}\label{eq: measurement operator}
y(t) = z_x(t,1) .
\end{equation}
Note that the output feedback stabilization of (\ref{eq: RD system 1}) in the linear case ($\mu = 0$) and with Dirichlet pointwise measurement was addressed in~\cite{katz2021finite} while the state-feedback of the nonlinear equation ($\mu \neq 0$) was reported in~\cite{katz2021regional}, both for PDE trajectories evaluated in $H^1$-norm.

We first introduce an homogeneous representation of (\ref{eq: RD system 1}); see~\cite{curtain2012introduction} for generalities on boundary control systems. To do so, consider the change of variable formula:
\begin{equation}\label{eq: change of variable}
w(t,x) = z(t,x) - (1-x) u(t) .
\end{equation}
This implies that
\begin{subequations}\label{eq: homogeneous RD system 1}
\begin{align}
& w_t = -w_{xxxx} - \lambda w_{xx} + b \dot{u} + r \label{eq: homogeneous RD system 1 - 1} \\
& w(t,0) = w(t,1) = w_{xx}(t,0) = w_{xx}(t,1) = 0 \label{eq: homogeneous RD system 1 - 2} \\
& w(0,x) = w_0(x) . \label{eq: homogeneous RD system 1 - 3}
\end{align}
\end{subequations}
where $b(x) = -(1-x)$, $r(t,x) = -\mu z(t,x) z_x(t,x) = - \mu ( w(t,x) - b(x) u(t) ) ( w_x(t,x) - u(t) )$, and $w_0(x) = z_0(x) - (1-x) u(0)$. Moreover, we have
\begin{equation}\label{eq: homogeneous measurement operator}
\tilde{y}(t) = w_x(t,1) = y(t) + u(t) .
\end{equation}

Let the unbounded operator $\mathcal{A} f = - f'''' - \lambda f''$ be defined on $D(\mathcal{A}) = \{ f \in H^4(0,1) \,:\, f(0)=f(1)=f''(0)=f''(1)=0 \}$. Hence (\ref{eq: homogeneous RD system 1}) can be written in abstract form as
\begin{subequations}\label{eq: abstract homogeneous RD system 1}
\begin{align}
& w_t = \mathcal{A}w + b \dot{u} + r \label{eq: abstract homogeneous RD system 1 - 1} \\
& w(0,x) = w_0(x) . \label{eq: abstract homogeneous RD system 1 - 2}
\end{align}
\end{subequations}
Introducing 
\begin{equation}\label{eq: definition phi_n}
\phi_n(x) = \sqrt{2}\sin(n\pi x) , \quad n \geq 1
\end{equation}
it is well-known that $\{ \phi_n \,:\, n \geq 1\}$ forms a Hilbert basis of $L^2(0,1)$. For any $f \in L^2(0,1)$ and any integer $N \geq 0$, we define $\mathcal{R}_N f = \sum_{n \geq N+1} \left< f,\phi_n \right> \phi_n$ hence $\Vert \mathcal{R}_N f \Vert_{L^2}^2 = \sum_{n \geq N+1} \left< f,\phi_n \right>^2$. It can be seen that $\mathcal{A}$ is self-adjoint and 
$$\mathcal{A}\phi_n = \sigma_n \phi_n , \quad \sigma_n = - n^4 \pi^4 + \lambda n^2 \pi^2 , \quad n \geq 1 .$$ 
Using arguments similar to \cite[Thm.~2.3.5]{curtain2012introduction} in the case where a finite number of eigenvalues may have a finite multiplicity greater than one, we deduce that $D(\mathcal{A}) = \{ f \in L^2(0,1) \,:\, \sum_{n \geq 1} \sigma_n^2 \left< f , \phi_n \right>^2 < \infty \} = \{ f \in L^2(0,1) \,:\, \sum_{n \geq 1} n^8 \left< f , \phi_n \right>^2 < \infty \}$. Moreover $\mathcal{A} f = \sum_{n \geq 1} \sigma_n \left< f , \phi_n \right> \phi_n$ for all $f \in D(\mathcal{A})$.

\begin{remark}\label{rmk: series expansion}
Owing to the definition of $\phi_n$ given by (\ref{eq: definition phi_n}), it is well known from the properties of Sturm-Liouville operators~\cite{renardy2006introduction} that $f = \sum_{n \geq 1} \left< f , \phi_n \right> \phi_n$ in $H^2$-norm for any $f \in H_0^1(0,1) \cap H^2(0,1)$. First, this implies that $f' = \sum_{n \geq 1} n\pi \left< f , \phi_n \right> \sqrt{2} \cos(n\pi\cdot)$ and $f'' = - \sum_{n \geq 1} n^2 \pi^2 \left< f , \phi_n \right> \sqrt{2} \sin(n\pi\cdot)$, both in $L^2$-norm. Therefore, we obtain that 
\begin{equation}\label{eq: H1 norm in function of the modes}
\Vert f' \Vert_{L^2}^2 = \sum_{n \geq 1} n^2 \pi^2 \left< f , \phi_n \right>^2 
, \quad 
\Vert f'' \Vert_{L^2}^2 = \sum_{n \geq 1} n^4 \pi^4 \left< f , \phi_n \right>^2
\end{equation}
for all $f \in H_0^1(0,1) \cap H^2(0,1)$. Second, from the continuous embedding $H^1(0,1) \subset L^\infty(0,1)$, we have $f(x) = \sum_{n \geq 1} \left< f , \phi_n \right> \phi_n(x)$ and $f'(x) = \sum_{n \geq 1} \left< f , \phi_n \right> \phi_n'(x)$ for all $x \in [0,1]$ and all $f \in H_0^1(0,1) \cap H^2(0,1)$.
\end{remark}

Introducing now the Fourier coefficients defined by $z_n(t) = \left< z(t,\cdot), \phi_n \right>$, $w_n(t) = \left< w(t,\cdot) , \phi_n \right>$, $b_n = \left< b , \phi_n \right> = - \frac{\sqrt{2}}{n\pi}$, and $r_n(t) = \left< r(t,\cdot) , \phi_n \right>$, the projection of (\ref{eq: abstract homogeneous RD system 1}) into $(\phi_n)_{n \geq 1}$ gives 
\begin{equation}\label{eq: RD system 1 spectral reduction 1}
\dot{w}_n (t)= \sigma_n w_n(t) + b_n \dot{u}(t) + r_n(t) .
\end{equation}
Moreover, from the change of variable formula (\ref{eq: change of variable}) we have that
\begin{equation}\label{eq: projected change of variable}
w_n(t) = z_n(t) + b_n u(t) .
\end{equation}
Hence, the combination of the two latter equations gives:
\begin{equation}\label{eq: RD system 1 spectral reduction z coordinate}
\dot{z}_n (t)= \sigma_n z_n(t) + \beta_n u(t) + r_n(t) .
\end{equation}
with 
$\beta_n = \sigma_n b_n = -\sqrt{2} n\pi ( -n^2\pi^2 + \lambda )$. 
Finally, considering the concept of classical solutions for which $w(t,\cdot) \in D(\mathcal{A})$, the system output (\ref{eq: homogeneous measurement operator}) is expressed as the series expansion (see Remark~\ref{rmk: series expansion}):
\begin{equation}\label{eq: RD system 1 spectral reduction 3}
\tilde{y}(t) = \sum_{n \geq 1} c_n w_n(t) 
\end{equation}
with $c_n = \phi_n'(1)$.

\subsection{Control strategy}

The control architecture adopted in this paper appeared in~\cite{lhachemi2021nonlinear} for reaction-diffusion PDEs and is widely inspired by the pioneer work~\cite{sakawa1983feedback}. We fix $\delta > 0$ the desired exponential decay rate for the closed-loop system trajectories. Let an integer $N_0 \geq 1$ be fixed such that $\sigma_n < - \delta < 0$ for all $n \geq N_0 + 1$. Introducing an arbitrary integer $N \geq N_0 + 1$, that will be specified later, the control strategy is described by the following dynamics: 
\begin{subequations}\label{eq: controller 2}
\begin{align}
\hat{w}_n(t) & = \hat{z}_n(t) + b_n u(t) ,\quad 1 \leq n \leq N \label{eq: controller 2 - 0} \\
\dot{\hat{z}}_n(t) & = \sigma_n \hat{z}_n(t) + \beta_n u(t) \nonumber \\
& \phantom{=}\; - l_n \left\{ \sum_{k = 1}^N c_k \hat{w}_k(t) - \tilde{y}(t) \right\} ,\quad 1 \leq n \leq N_0 \label{eq: controller 2 - 1} \\
\dot{\hat{z}}_n(t) & = \sigma_n \hat{z}_n(t) + \beta_n u(t) ,\quad N_0+1 \leq n \leq N \label{eq: controller 2 - 2} \\
u(t) & = \sum_{n=1}^{N_0} k_n \hat{z}_n(t) \label{eq: controller 2 - 3}
\end{align}
\end{subequations}
where $l_n,k_n \in\R$ are the observer and feedback gains, respectively. 

\begin{remark}\label{rmk: well-posedness}
For any initial condition $w_0 \in D(\mathcal{A})$ and any $\hat{z}_n(0) \in\R$, the well-posedness of the closed-loop system composed of (\ref{eq: RD system 1}), rewritten under the homogeneous representation (\ref{eq: homogeneous RD system 1}), and (\ref{eq: controller 2}) is a straightforward consequence of the application of \cite[Thm.~6.3.1]{pazy2012semigroups}. Moreover, owing to the measured output (\ref{eq: RD system 1 spectral reduction 3}) written in homogeneous coordinates and from the argument of \cite[Thm.~6.3.3]{pazy2012semigroups}, the system trajectories will be defined for all $t \geq 0$ if it can be shown that $ \Vert w_{xx}(t,\cdot) \Vert_{L^2}^2 = \sum_{n \geq 1} n^4 \pi^4 w_n(t)^2$ remains bounded on any time interval of finite length.
\end{remark}

\subsection{Reduced order model for stability analysis}

The statement of the main result and the associated proof requires the introduction of a reduced order model that captures the $N$ first modes of the PDE (\ref{eq: RD system 1}), as well as the dynamics of the controller (\ref{eq: controller 2}). We define the scaled estimation $\tilde{z}_n = \hat{z}_n / n^4$, the error of observation $e_n = z_n - \hat{z}_n$, the scaled error of observation $\tilde{e}_n = n^2 e_n$, and the vectors 
$\hat{Z}^{N_0} = \begin{bmatrix} \hat{z}_1 & \ldots & \hat{z}_{N_0} \end{bmatrix}^\top$, 
$E^{N_0} = \begin{bmatrix} e_1 & \ldots & e_{N_0} \end{bmatrix}^\top$, $\tilde{Z}^{N-N_0} = \begin{bmatrix} \tilde{z}_{N_0+1} & \ldots & \tilde{z}_{N} \end{bmatrix}^\top$,
$\tilde{E}^{N-N_0} = \begin{bmatrix} \tilde{e}_{N_0 +1} & \ldots & \tilde{e}_{N} \end{bmatrix}^\top$, 
$R_1 = \begin{bmatrix} r_1 & \ldots & r_{N_0} \end{bmatrix}^\top$,
$\tilde{R}_2 = \begin{bmatrix} (N_0+1)^2 r_{N_0 +1} & \ldots & N^2 r_{N} \end{bmatrix}^\top$, 
$R = \mathrm{col}(R_1,\tilde{R}_2)$.
We also introduce the matrices defined by 
$A_0 = \mathrm{diag}(\sigma_1,\ldots,\sigma_{N_0})$,
$A_1 = \mathrm{diag}(\sigma_{N_0+1},\ldots,\sigma_N)$,
$\mathfrak{B}_0 = \begin{bmatrix} \beta_1 & \ldots & \beta_{N_0} \end{bmatrix}^\top$,
$\tilde{\mathfrak{B}}_1 = \begin{bmatrix} \frac{\beta_{N_0+1}}{(N_0+1)^4} & \ldots & \frac{\beta_N}{N^4} \end{bmatrix}^\top$,
$C_0 = \begin{bmatrix} c_1 & \ldots & c_{N_0} \end{bmatrix}$, 
$\tilde{C}_1 = \begin{bmatrix} \frac{c_{N_0+1}}{(N_0+1)^2} & \ldots & \frac{c_{N}}{N^2} \end{bmatrix}$, 
$K = \begin{bmatrix} k_1 & \ldots & k_{N_0} \end{bmatrix}$,
$L = \begin{bmatrix} l_1 & \ldots & l_{N_0} \end{bmatrix}^\top$.
Similarly to \cite{lhachemi2020finite,lhachemi2021nonlinear}, the benefit of introducing the scaled quantities $\tilde{z}_n$ and $\tilde{e}_n$ is to ensure that $\Vert \tilde{C}_1 \Vert = O(1)$ and $\Vert \tilde{\mathfrak{B}}_1 \Vert = O(1)$ as $N \rightarrow + \infty$; this will be key in the proof of the theorems to ensure the systematic nature of the procedure. With these definitions, we deduce from (\ref{eq: RD system 1 spectral reduction 1}-\ref{eq: controller 2}) that
\begin{align*}
\dot{\hat{Z}}^{N_0} & = (A_0 + \mathfrak{B}_0 K) \hat{Z}^{N_0} + LC_0 E^{N_0} + L \tilde{C}_1 \tilde{E}^{N-N_0} + L \zeta \\
\dot{E}^{N_0} & = (A_0 - LC_0) E^{N_0} - L \tilde{C}_1 \tilde{E}^{N-N_0} - L \zeta + R_1 \\
\dot{\tilde{Z}}^{N-N_0} & = A_1 \tilde{Z}^{N-N_0} + \tilde{\mathfrak{B}}_1 K \hat{Z}^{N_0} \\
\dot{\tilde{E}}^{N-N_0} & = A_1 \tilde{E}^{N-N_0} + \tilde{R}_2 \\
u & = K \hat{Z}^{N_0} .
\end{align*}
where 
\begin{equation}\label{eq: residue of measurement}
\zeta = \sum_{n \geq N+1} \phi_n'(1) w_n
\end{equation} 
stands for the residue of measurement. Defining the state vector
\begin{equation*}
X = \mathrm{col}\left( \hat{Z}^{N_0} , E^{N_0} , \tilde{Z}^{N-N_0} , \tilde{E}^{N-N_0} \right) ,
\end{equation*}
we obtain the reduced order model described by
\begin{equation}\label{eq: RD system 2 dynamics truncated model}
\dot{X} = F X + \mathcal{L} \zeta + G R
\end{equation}
where
\begin{equation*}
F = \begin{bmatrix}
A_0 + \mathfrak{B}_0 K & LC_0 & 0 & L\tilde{C}_1 \\
0 & A_0 - L C_0 & 0 & - L\tilde{C}_1 \\
\tilde{\mathfrak{B}}_1 K & 0 & A_1 & 0 \\
0 & 0 & 0 & A_1
\end{bmatrix}
, \quad 
G = \begin{bmatrix}
0 & 0 \\ I & 0 \\ 0 & 0 \\ 0 & I
\end{bmatrix} 
\end{equation*}
and $\mathcal{L} = \mathrm{col}(L,-L,0,0)$. Introducing finally the augmented vector $\tilde{X} = \mathrm{col}(X,\zeta,R)$, the control input is given by 
\begin{equation}\label{eq: u in function of X}
u = \tilde{K} X , \qquad \dot{u} = K \dot{\hat{Z}}^{N_0} = E\, \mathrm{col}(X,\zeta) = \tilde{E} \tilde{X} .
\end{equation}
where $\tilde{K} = \begin{bmatrix} K & 0 & 0 & 0 \end{bmatrix}$, $E = K \begin{bmatrix} A_0 + \mathfrak{B}_0 K & LC_0 & 0 & L\tilde{C}_1 & L \end{bmatrix}$, and $\tilde{E} = K \begin{bmatrix} A_0 + \mathfrak{B}_0 K & LC_0 & 0 & L\tilde{C}_1 & L & 0 \end{bmatrix} = K \begin{bmatrix} E & 0 \end{bmatrix}$.

\begin{remark}\label{rmk: Kalman condition}
Since $A_0$ is diagonal, the pair $(A_0,\mathfrak{B}_0)$ satisfies the Kalman condition if and only if $\beta_n \neq 0$ for all $1 \leq n \leq N_0$ and for any $1 \leq n \neq m \leq N_0$ we have $\sigma_n \neq \sigma_m$. Since $\beta_n = \sigma_n b_n$ with $b_n = - \frac{\sqrt{2}}{n\pi} \neq 0$, the former condition is equivalent to $\sigma_n \neq 0$, i.e., there does not exist an integer $n \geq 1$ so that $\lambda = n^2 \pi^2$. Moreover, the existence of integers $1 \leq n < m$ so that $\sigma_n = \sigma_m$ is equivalent to the existence of integers $1 \leq n < m$ so that $\lambda = (n^2+m^2)\pi^2$. In that case $\sigma_n = \sigma_m = n^2 m^2 \pi^4 > 0$, hence we necessarily have $1 \leq n < m \leq N_0$, by definition of the integer $N_0$. For $\lambda > 0$, this shows that the pair $(A_0,\mathfrak{B}_0)$ satisfies the Kalman condition if and only if $\lambda \notin \Lambda = \Lambda_1 \cup \Lambda_2$ where 
\begin{align*}
\Lambda_1 & = \{ n^2 \pi^2 \,:\, n\in\N^* \} , \\
\Lambda_2 & = \{ (n^2+m^2)\pi^2 \,:\, n,m\in\N^*,\,  n < m \} .
\end{align*}
Similarly, the pair $(A_0,C_0)$ satisfies the Kalman condition if and only if $\phi_n'(1) \neq 0$ for all $1 \leq n \leq N_0$ and for any $1 \leq n \neq m \leq N_0$ we have $\sigma_n \neq \sigma_m$. Since $\phi_n'(1) = (-1)^n \sqrt{2} n \pi \neq 0$, the pair $(A_0,C_0)$ satisfies the Kalman condition if and only if $\lambda \notin \Lambda_2$. 
\end{remark}

\subsection{Local stability assessment}

The case $\lambda\notin\Lambda$ allows to fix for good the gains $K\in\R^{1 \times N_0}$ and $L\in\R^{N_0}$ such that $A_0 + \mathfrak{B}_0 K$ and $A_0 - L C_0$ are Hurwitz with eigenvalues having a real part strictly less than $-\delta<0$. It remains to fix the dimension $N \geq N_0+1$ of the observer. This can be done using the following result.

\begin{theorem}\label{thm1}
Let $\lambda > 0$ with $\lambda \notin \Lambda$ and $\mu \in\R$. Let $\delta > 0$ and $N_0 \geq 1$ be such that $\sigma_n < - \delta$ for all $n \geq N_0 + 1$. Let $K\in\R^{1 \times N_0}$ and $L\in\R^{N_0}$ be such that $A_0 + \mathfrak{B}_0 K$ and $A_0 - L C_0$ are Hurwitz with eigenvalues that have a real part strictly less than $-\delta<0$. For a given $N \geq N_0 +1$, assume that there exist a symmetric positive definite $P\in\R^{2N \times 2N}$ and positive real numbers $\alpha > 1/\pi^4$ and $\beta,\gamma > 0$ such that 
\begin{equation}\label{eq: thm1 constraints}
\Theta_1(\delta) \prec 0 , \quad
\Theta_2(\delta) < 0 , \quad
\Theta_3 \geq 0
\end{equation}
where
\begin{align*}
\Theta_1(\delta) & = \begin{bmatrix}
F^\top P + P F + 2 \delta P & P \mathcal{L}  \\
\mathcal{L}^\top P & -\beta
\end{bmatrix} 
+ \alpha\gamma \Vert \mathcal{R}_N b \Vert_{L^2}^2 E^\top E \\
\Theta_2(\delta) & = 2 \gamma \left( \sigma_{N+1} + \delta + \frac{(N+1)^4}{\alpha} \right) + \beta M_\phi \\
\Theta_3 & = (N+1)^2 \left( \pi^4 - \frac{1}{\alpha} \right) - \lambda \pi^2 
\end{align*}
where $M_\phi = 2\pi^2 \sum_{n \geq N+1} \frac{1}{n^2} = 2\pi^2 \left( \frac{\pi^2}{6} - \sum_{n=1}^N \frac{1}{n^2} \right)$. Then, considering the closed-loop system composed of the plant (\ref{eq: RD system 1}) with the system output (\ref{eq: measurement operator}) and the controller (\ref{eq: controller 2}), there exist $C,M > 0$ such that for any initial conditions $z_0 \in D(\mathcal{A})$ with $\Vert z_0 \Vert_{H^2} < C$, the system trajectory with null initial condition of the observer ($\hat{z}_n(0) = 0$ for all $1 \leq n \leq N$) satisfies 
\begin{equation}\label{eq: thm1 stability estimate}
\Vert z(t,\cdot)\Vert_{H^2}^2 + \sum_{n=1}^N \hat{z}_n(t)^2 \leq M e^{-2\delta t} \Vert z_0 \Vert_{H^2}^2
\end{equation} 
for all $t \geq 0$. Moreover, the constraints (\ref{eq: thm1 constraints}) are feasible when $N$ is selected large enough.
\end{theorem}

\begin{proof}
From (\ref{eq: thm1 constraints}) and by a continuity argument, we fix $\kappa > \delta$ so that $\Theta_1(\kappa) \prec 0$ and $\Theta_2(\kappa) \leq 0$. We introduce $\rho > 0$ selected large enough so that $\tilde{\Theta}_1(\kappa) \preceq 0$ where 
\begin{equation*}
\tilde{\Theta}_1(\kappa) = \begin{bmatrix}
F^\top P + P F + 2 \kappa P & P \mathcal{L} & P G  \\
\mathcal{L}^\top P & -\beta & 0 \\
G^\top P & 0 & - \rho I
\end{bmatrix} 
+ \alpha\gamma \Vert \mathcal{R}_N b \Vert_{L^2}^2 \tilde{E}^\top \tilde{E}.
\end{equation*}
Consider the Lyapunov functional defined for $X\in\R^{2N}$ and $w \in D(\mathcal{A})$ by
\begin{equation}\label{eq: def of V}
V(X,w) = X^\top P X + \gamma \sum_{n \geq N+1} n^4 w_n^2 .
\end{equation}
The choice of this Lyapunov functional is essentially motivated by the fact that we need to ensure that $\sum_{n \geq 1} n^4 w_n(t)^2$ remains bounded on any time interval of finite length in order to guarantee the existence of the system trajectories for all $t \geq 0$, see Remark~\ref{rmk: well-posedness}. Computing the time derivative of $V$ along the system trajectories (\ref{eq: RD system 1 spectral reduction 1}) and (\ref{eq: RD system 2 dynamics truncated model}), we infer that
\begin{align*}
& \dot{V} + 2 \kappa V = X^\top \{ F^\top P + P F + 2 \kappa P \} X + 2 X^\top P \mathcal{L} \zeta \\
& \quad  + 2 X^\top P G R + 2\gamma\sum_{n \geq N+1} n^4 w_n \{ (\sigma_n+\kappa) w_n + b_n \dot{u} + r_n \} .
\end{align*} 
The use of Young's inequality implies that 
\begin{align*}
2\sum_{n \geq N+1} n^4 w_n b_n \dot{u} & \leq \frac{1}{\alpha} \sum_{n \geq N+1} n^8 w_n^2 + \alpha \Vert \mathcal{R}_N b \Vert_{L^2}^2 \dot{u}^2 , \\
2\sum_{n \geq N+1} n^4 w_n r_n & \leq \frac{1}{\alpha} \sum_{n \geq N+1} n^8 w_n^2 + \alpha \Vert \mathcal{R}_N r \Vert_{L^2}^2 .
\end{align*}
Hence, using (\ref{eq: u in function of X}), we infer that
\begin{align*}
& \dot{V} + 2 \kappa V \\
& \leq X^\top \{ F^\top P + P F + 2 \kappa P \} X + 2 X^\top P \mathcal{L} \zeta + 2 X^\top P G R \\
& \phantom{\leq}\; + \alpha\gamma \Vert \mathcal{R}_N b \Vert_{L^2}^2 \tilde{X}^\top \tilde{E}^\top \tilde{E} \tilde{X} + 2\gamma\sum_{n \geq N+1} n^4 \left\{ \sigma_n+\kappa + \frac{n^4}{\alpha} \right\} w_n^2 \\
& \phantom{\leq}\; + \alpha\gamma \Vert \mathcal{R}_N r \Vert_{L^2}^2 .
\end{align*} 
From the definition of $\zeta$ given by (\ref{eq: residue of measurement}), the use of Cauchy-Schwarz inequality gives $\zeta^2 \leq M_\phi \sum_{n \geq N+1} n^4 w_n^2$. This implies that
\begin{align*}
& \dot{V} + 2 \kappa V \\
& \leq \tilde{X}^\top \tilde{\Theta}_1(\kappa) \tilde{X} 
+ \sum_{n \geq N+1} n^4 \Gamma_n w_n^2 
+ \alpha \gamma \Vert \mathcal{R}_N r \Vert_{L^2}^2 + \rho \Vert R \Vert^2 \\
& \leq \tilde{X}^\top \tilde{\Theta}_1(\kappa) \tilde{X} 
+ \sum_{n \geq N+1} n^4 \Gamma_n w_n^2 + \max ( \alpha \gamma , \rho N^4 ) \Vert r \Vert_{L^2}^2 
\end{align*}
with 
\begin{align*}
\Gamma_n 
& = 2\gamma  \left( \sigma_n + \kappa + \frac{n^4}{\alpha} \right) + \beta M_\phi \\
& = - 2\gamma n^2 \left\{ n^2 \left( \pi^4 - \frac{1}{\alpha} \right) - \lambda \pi^2 \right\} + 2\gamma\kappa + \beta M_\phi \\
& \leq - 2 \gamma n^2 \Theta_3 + 2\gamma\kappa + \beta M_\phi \\
& \leq -2 \gamma (N+1)^2 \Theta_3 + 2\gamma\kappa + \beta M_\phi = \Theta_2(\kappa)
\end{align*}
for all $n \geq N+1$, where we have used that $\alpha > 1/\pi^4$ and $\Theta_3 \geq 0$. Defining $C_1 = \max ( \alpha \gamma , \rho N^4 ) > 0$, and recalling that $\tilde{\Theta}_1(\kappa) \preceq 0$ and $\Theta_2(\kappa) \leq 0$, we obtain that 
$$\dot{V}+2\kappa V \leq C_1 \Vert r \Vert_{L^2}^2 .$$ 
We now study the term $\Vert r \Vert_{L^2}^2$ where $r = -\mu z z_x$. We focus on the case $\mu \neq 0$ (if $\mu = 0$ then $r=0$ and the claimed result easily follows) by computing 
$\Vert r(t,\cdot) \Vert_{L^2}^2
= \mu^2 \int_0^1 \vert z(t,x) \vert^2 \vert z_x(t,x) \vert^2 \,\mathrm{d}x
\leq \mu^2 \Vert z(t,\cdot) \Vert_{L^\infty}^2 \Vert z_x(t,\cdot) \Vert_{L^2}^2$.
Since $z(t,1) = 0$, we have $z(t,x) = \int_1^x z_x(t,s) \,\mathrm{d}s$, implying that $\vert z(t,x) \vert \leq \int_0^1 \vert z_x(t,s) \vert \,\mathrm{d}s \leq \Vert z_x(t,\cdot) \Vert_{L^2}$ hence $\Vert z(t,\cdot) \Vert_{L^\infty} \leq \Vert z_x(t,\cdot) \Vert_{L^2}$. Therefore, we deduce that 
$$\Vert r(t,\cdot) \Vert_{L^2}^2 \leq \mu^2 \Vert z_x(t,\cdot) \Vert_{L^2}^4 . $$ 
Owing to (\ref{eq: change of variable}) we have $z_x(t,x) = w_x(t,x) - u(t)$ hence $\Vert z_x(t,\cdot) \Vert_{L^2} \leq \Vert w_x(t,\cdot) \Vert_{L^2} + \vert u(t) \vert$. Based on (\ref{eq: H1 norm in function of the modes}), we have $\Vert w_x(t,\cdot) \Vert_{L^2}^2 = \sum_{n \geq 1} n^2 \pi^2 w_n(t)^2 \leq \pi^2 \sum_{n=1}^N n^2 w_n(t)^2 + \frac{\pi^2}{\gamma (N+1)^2} V(t)$. Moreover from (\ref{eq: projected change of variable}) we have $w_n = z_n + b_n u = e_n + \hat{z}_n + b_n u$. This implies that $\Vert w_x(t,\cdot) \Vert_{L^2}^2 \leq C_1 V(t) + C_2 u(t)^2$ for some constants $C_1,C_2 > 0$. Recalling that the input $u$ is expressed by (\ref{eq: u in function of X}), we have $\vert u(t) \vert^2 \leq \Vert \tilde{K} \Vert^2 \Vert X(t) \Vert^2 \leq \frac{\Vert \tilde{K} \Vert^2}{\lambda_m(P)} V(t)$. Putting the estimates together we obtain that 
$$\Vert z_x(t,\cdot) \Vert_{L^2}^2 \leq 2 \Vert w_x(t,\cdot) \Vert_{L^2}^2 + 2 \vert u(t) \vert^2 \leq C_3 V(t)$$ 
for some constant $C_3 > 0$, which in turns gives $\dot{V}+2\kappa V \leq C_4 V^2 $ for some constant $C_4 > 0$. Recalling that $0 < \delta < \kappa$, we have 
$$\dot{V}+2\delta V \leq - \left\{ 2(\kappa-\delta) - C_3 V \right\} V.$$ 
Hence, if we select the initial condition so that $V(0) < C = \frac{2(\kappa-\delta)}{C_4}$, we conclude that $\dot{V}+2\kappa V \leq 0$ on the maximal interval of existence of the system trajectories. From the definition (\ref{eq: def of V}) of $V$, this implies that $\sum_{n \geq 1} n^4 w_n(t)^2$ remains bounded on any time interval of finite length. Hence the system trajectory is well defined for all $t \geq 0$ with $\dot{V}+2\kappa V \leq 0$. The claimed stability estimate (\ref{eq: thm1 stability estimate}) easily follows from the change of variable (\ref{eq: change of variable}), the identities (\ref{eq: H1 norm in function of the modes}), and the definition (\ref{eq: def of V}) of $V$.

We conclude the proof by showing that the constraints (\ref{eq: thm1 constraints}) are feasible when $N$ is selected large enough. To do so we note that (i) $A_0 + \mathfrak{B}_0 K + \delta I$ and $A_0 - L C_0 + \delta I$ are Hurwitz; (ii) $\Vert e^{(A_1+\delta I) t} \Vert \leq e^{-\kappa_0 t}$ for all $t \geq 0$ with some positive constant $\kappa_0 > 0$ independent of $N$; (iii) $\Vert L \tilde{C}_1 \Vert \leq \Vert L \Vert \Vert \tilde{C}_1 \Vert$ and $\Vert \tilde{\mathfrak{B}}_1 K \Vert \leq \Vert \tilde{\mathfrak{B}}_1 \Vert \Vert K \Vert$ where $\Vert L \Vert$ and $\Vert K \Vert$ are independent of $N$ while $\Vert \tilde{\mathfrak{B}}_1 \Vert = O(1)$ and $\Vert \tilde{C}_1 \Vert = O(1)$ a $N \rightarrow + \infty$. The application of the Lemma reported in appendix of \cite{lhachemi2020finite} to the matrix $F+\delta I$ shows that the unique solution $P \succ 0$ to the Lyapunov equation $F^\top P + P F + 2 \delta P = -I$ satisfies $\Vert P \Vert = O(1)$ as $N \rightarrow + \infty$. We fix $\alpha > 1/\pi^4$ arbitrarily. We define $\beta=\sqrt{N}$ and $\gamma = 1/N$. This implies that $\Theta_2 \rightarrow - \infty$ and $\Theta_3 \rightarrow + \infty$ as $N \rightarrow + \infty$. Moreover, since $\Vert \mathcal{L} \Vert = \sqrt{2} \Vert L \Vert$ is independent of $N$, and $\Vert P \Vert = O(1)$ and $\Vert E \Vert = O(1)$ as $N \rightarrow + \infty$, the application of the Schur complement implies that $\Theta_1 \prec 0$ for $N$ selected to be large enough. This completes the proof.
\end{proof}

\begin{remark}
For a given order $N \geq N_0+1$ of the observer, the control design constraints (\ref{eq: thm1 constraints}) are nonlinear w.r.t. the decision variables $P \succ 0$, $\alpha > 1/\pi^4$, and $\beta,\gamma>0$. However, if one fix arbitrarily the value of $\alpha > 1/\pi^4$, we obtain a LMI formulation of the constraints. As shown in the proof, this latter formulation remains feasible when selecting $N$ to be large enough.
\end{remark}

\section{Remedies for the case $\lambda \in \Lambda$}\label{sec3}

From Remark~\ref{rmk: Kalman condition}, the configuration $\lambda \in \Lambda$ does not allow the application of Theorem~\ref{thm1} for $\lambda \in \Lambda$ due to the loss of controllability or/and observability properties. We discuss here how to change the actuation/sensing scheme in order to achieve the local exponential stabilization of the KSE when $\lambda \in \Lambda$.

\subsection{Remedy for the case $\lambda \in \Lambda_1 \backslash \Lambda_2$}

In the case $\lambda \in \Lambda_1 \backslash \Lambda_2$ we have: (i) there exists a unique $n_0 \in \N^*$ so that $\lambda = n_0^2 \pi^2$, which gives in particular $\sigma_{n_0} = 0$; and (ii) $\sigma_n \neq \sigma_m$ for all $n \neq m$. Hence, all the modes are observable using the system output (\ref{eq: measurement operator}). However, in the actuation scheme of (\ref{eq: RD system 1}), there is one (and only one) uncontrollable eigenvalue: $\sigma_{n_0} = 0$. In order to circumvent this issue for achieving the stabilization of the KSE, we propose to modify the actuation scheme as follow:
\begin{subequations}\label{eq: RD system 2}
\begin{align}
& z_t + z_{xxxx} + \lambda z_{xx} + \mu z z_x = 0 \label{eq: RD system 2 - 1} \\
& z_{xx}(t,0) = u(t) \label{eq: RD system 2 - 2} \\
& z(t,0) = z(t,1) = z_{xx}(t,1) = 0 \label{eq: RD system 2 - 3} \\
& z(0,x) = z_0(x) \label{eq: RD system 2 - 4}
\end{align}
\end{subequations}
along with the system output (\ref{eq: measurement operator}). Based on the change of variable formula
\begin{equation}\label{eq: change of variable 2}
w(t,x) = z(t,x) + \frac{1}{6} x (1-x)^3 u(t) 
\end{equation}
we infer that
\begin{align*}
& w_t = -w_{xxxx} - \lambda w_{xx} + a u + b \dot{u} + r \\
& w(t,0) = w(t,1) = w_{xx}(t,0) = w_{xx}(t,1) = 0 \\
& w(0,x) = w_0(x)  
\end{align*}
where $b(x) = \frac{1}{6} x (1-x)^3$, $a(x) = b''''(x) + \lambda b''(x)$, $r(t,x) = -\mu z(t,x) z_x(t,x)$, and $w_0(x) = z_0(x) + \frac{1}{6} x (1-x)^3 u(0)$. Moreover, we have
\begin{equation*}
\tilde{y}(t) = w_x(t,1) = y(t) .
\end{equation*}

Considering $\phi_n$ defined by (\ref{eq: definition phi_n}), along with the coefficients the Fourier coefficients defined by $z_n(t) = \left< z(t,\cdot), \phi_n \right>$, $w_n(t) = \left< w(t,\cdot) , \phi_n \right>$, $a_n = \left< a , \phi_n \right>$, $b_n = \left< b , \phi_n \right>$, and $r_n(t) = \left< r(t,\cdot) , \phi_n \right>$, we infer that 
\begin{equation*}
\dot{w}_n (t)= \sigma_n w_n(t) + a_n u(t) + b_n \dot{u}(t) + r_n(t) . 
\end{equation*}
The change of variable formula (\ref{eq: change of variable 2}) gives (\ref{eq: projected change of variable}). Thus the combination of the latter equations implies (\ref{eq: RD system 1 spectral reduction z coordinate})
with 
$$\beta_n = a_n + \sigma_n b_n = - \phi_n'(0) = - \sqrt{2} n \pi .$$ 
We obtain that $\beta_n \neq 0$ for all $n \geq 1$. Assuming that $\lambda > 0$ is so that $\lambda\notin\Lambda_2$ and applying the same procedure that the one reported in the previous section, we deduce that the pair $(A_0,\mathfrak{B}_0)$ satisfies the Kalman condition. This allows the statement of the following theorem whose proof is analogous to the one of Theorem~\ref{thm1}.

\begin{theorem}\label{thm2}
Let $\lambda > 0$ with $\lambda \notin \Lambda_2$ and $\mu \in\R$. Let $\delta > 0$ and $N_0 \geq 1$ be such that $\sigma_n < - \delta$ for all $n \geq N_0 + 1$. Let $K\in\R^{1 \times N_0}$ and $L\in\R^{N_0}$ be such that $A_0 + \mathfrak{B}_0 K$ and $A_0 - L C_0$ are Hurwitz with eigenvalues that have a real part strictly less than $-\delta<0$. For a given $N \geq N_0 +1$, assume that there exist a symmetric positive definite $P\in\R^{2N \times 2N}$ and positive real numbers $\alpha > 3/(2\pi^4)$ and $\beta,\gamma > 0$ such that 
\begin{equation}\label{eq: thm2 constraints}
\Theta_1 \prec 0 , \quad
\Theta_2 < 0 , \quad
\Theta_3 \geq 0
\end{equation}
where
\begin{align*}
\Theta_1 & = \begin{bmatrix}
F^\top P + P F + 2 \delta P + \alpha\gamma \Vert \mathcal{R}_N a \Vert_{L^2}^2 \tilde{K}^\top \tilde{K} & P \mathcal{L}  \\
\mathcal{L}^\top P & -\beta
\end{bmatrix}  \\
& \phantom{=}\; + \alpha\gamma \Vert \mathcal{R}_N b \Vert_{L^2}^2 E^\top E \\
\Theta_2 & = 2 \gamma \left( \sigma_{N+1} + \delta + \frac{3(N+1)^4}{2\alpha} \right) + \beta M_\phi \\
\Theta_3 & = (N+1)^2 \left( \pi^4 - \frac{3}{2\alpha} \right) - \lambda \pi^2 
\end{align*}
where $M_\phi = 2\pi^2 \sum_{n \geq N+1} \frac{1}{n^2} = 2\pi^2 \left( \frac{\pi^2}{6} - \sum_{n=1}^N \frac{1}{n^2} \right)$. Then, considering the closed-loop system composed of the plant (\ref{eq: RD system 2}) with the system output (\ref{eq: measurement operator}) and the controller (\ref{eq: controller 2}), there exist $C,M > 0$ such that for any initial conditions $z_0 \in D(\mathcal{A})$ with $\Vert z_0 \Vert_{H^2} < C$, the system trajectory with null initial condition of the observer ($\hat{z}_n(0) = 0$ for all $1 \leq n \leq N$) satisfies 
\begin{equation}\label{eq: thm2 stability estimate}
\Vert z(t,\cdot)\Vert_{H^2}^2 + \sum_{n=1}^N \hat{z}_n(t)^2 \leq M e^{-2\delta t} \Vert z_0 \Vert_{H^2}^2
\end{equation} 
for all $t \geq 0$. Moreover, the constraints (\ref{eq: thm2 constraints}) are feasible when $N$ is selected large enough.
\end{theorem}

\subsection{Remedy for the case $\lambda \in \Lambda_2$}

The case $\lambda \in \Lambda_2$ is the most stringent to deal with because it implies the existence of integers $1 \leq n < m$ so that  $\sigma_n = \sigma_m$. In this case, $\lambda = (n^2 + m^2) \pi^2$ and $\sigma_n = \sigma_m = n^2 m^2 \pi^4 > 0$.  From there, it is easy to see that the eigenvalue $\sigma_n = \sigma_m$ is actually of multiplicty 2. Hence, owing to the diagonal structure of the unbounded operator $\mathcal{A}$, the plant has at least one unstable eigenvalue with multiplicity 2 which is not controllable (resp. observable) in the context of a single input (resp. output). Note there are situations for which there exist several eigenvalues of multiplicty 2. For example, in the case $\lambda = 65 \pi^2$ we have $\sigma_1 = \sigma_8 = 64 \pi^4$ and $\sigma_4 = \sigma_7 = 784 \pi^4$.

Based on these elements, the treatment of the case $\lambda\in\Lambda_2$ requires the increase of the number of control inputs and measured outputs, yielding a multi-input multi-output control design problem. In this study we consider the case :
\begin{subequations}\label{eq: RD system 3}
\begin{align}
& z_t + z_{xxxx} + \lambda z_{xx} + \mu z z_x = 0 \label{eq: RD system 3 - 1} \\
& z(t,0) = u_1(t) , \quad z_{xx}(t,0) = u_2(t) \label{eq: RD system 3 - 2} \\
& z(t,1) = z_{xx}(t,1) = 0 \label{eq: RD system 3 - 3} \\
& z(0,x) = z_0(x) \label{eq: RD system 3 - 4}
\end{align}
\end{subequations}
along with the system outputs
\begin{equation}\label{eq: measurement operator 2}
y(t) = \begin{bmatrix} z_x(t,1) & z_x(t,\xi) \end{bmatrix}^\top 
\end{equation}
for some $\xi\in[0,1)$ to be selected. Using the change of variable:
\begin{equation}\label{eq: change of variable 3}
w(t,x) = z(t,x) - (1-x) u_1(t) + \frac{1}{6} x (1-x)^3 u_2(t) ,
\end{equation}
we infer that
\begin{align*}
& w_t = -w_{xxxx} - \lambda w_{xx} + a_2 u_2 + b_1 \dot{u}_1 + b_2 \dot{u}_2 + r \\
& w(t,0) = w(t,1) = w_{xx}(t,0) = w_{xx}(t,1) = 0 \\
& w(0,x) = w_0(x) . 
\end{align*}
where $b_1(x) = -(1-x)$, $b_2(x) = \frac{1}{6} x (1-x)^3$, $a_2(x) = b_2''''(x) + \lambda b_2''(x)$, $r(t,x) = -\mu z(t,x) z_x(t,x)$, and $w_0(x) = z_0(x) - (1-x) u_1(0) + \frac{1}{6} x (1-x)^3 u_2(0)$. Moreover, we have
\begin{equation*}
\tilde{y}(t) 
= \begin{bmatrix} w_x(t,1) \\ w_x(t,\xi) \end{bmatrix}
= y(t) + \begin{bmatrix} u_1(t) \\ u_1(t) + b_2'(\xi) u_2(t) \end{bmatrix} 
\end{equation*}
which can be expressed by (\ref{eq: RD system 1 spectral reduction 3}) with $c_n = \begin{bmatrix} c_{1,n} & c_{2,n} \end{bmatrix}^\top = \begin{bmatrix} \phi_n'(1) & \phi_n'(\xi) \end{bmatrix}^\top$. Where $\phi_n$ is still defined by (\ref{eq: definition phi_n}). We introduce the Fourier coefficients defined by $z_n(t) = \left< z(t,\cdot), \phi_n \right>$, $w_n(t) = \left< w(t,\cdot) , \phi_n \right>$, $a_{2,n} = \left< a_2 , \phi_n \right>$, $b_{1,n} = \left< b_1 , \phi_n \right> = - \frac{\sqrt{2}}{n\pi}$, $b_{2,n} = \left< b_2 , \phi_n \right>$, and $r_n(t) = \left< r(t,\cdot) , \phi_n \right>$. We infer that 
\begin{equation*}
\dot{w}_n (t)= \sigma_n w_n(t) + a_{2,n} u_2(t) + b_{1,n} \dot{u}_1(t) + b_{2,n} \dot{u}_2(t) + r_n(t) . 
\end{equation*}
The use of the change of variable formula (\ref{eq: change of variable 3}) gives 
\begin{align*}
w_n(t) & = z_n(t) + b_{1,n} u_1(t) + b_{2,n} u_2(t) \\
& = z_n(t) + b_{n} u(t)
\end{align*}
where $u(t) = \begin{bmatrix} u_1(t) & u_2(t) \end{bmatrix}^\top$ and $b_n = \begin{bmatrix} b_{1,n} & b_{2,n} \end{bmatrix}$. This implies that
\begin{align*}
\dot{z}_n (t) & = \sigma_n z_n(t) + \beta_{1,n} u_1(t) + \beta_{2,n} u_2(t) + r_n(t) \\
& = \sigma_n z_n(t) + \beta_{n} u(t) + r_n(t) 
\end{align*}
where $\beta_{1,n} = \sigma_n b_{1,n} = -\sqrt{2} n\pi ( -n^2\pi^2 + \lambda )$, $\beta_{2,n} = a_{2,n} + \sigma_n b_{2,n} = - \sqrt{2} n \pi$, and $\beta_n = \begin{bmatrix} \beta_{1,n} & \beta_{2,n} \end{bmatrix}$. Considering feedback gains $k_n \in\R^2$ and observer gains $l_n \in\R^{1 \times 2}$, the controller architecture is still given by (\ref{eq: controller 2}).  So we define the different matrices as in the previous sections except: 
\begin{equation*}
\mathfrak{B}_0 = \begin{bmatrix} \beta_{1,1} & \beta_{2,1} \\ \vdots & \vdots \\ \beta_{1,N_0} & \beta_{2,N_0} \end{bmatrix}
, \quad
\tilde{\mathfrak{B}}_1 = \begin{bmatrix} \frac{\beta_{1,N_0 + 1}}{(N_0+1)^4} & \frac{\beta_{2,N_0 + 1}}{(N_0+1)^4} \\ \vdots & \vdots \\ \frac{\beta_{1,N}}{N^4} & \frac{\beta_{2,N}}{N^4} \end{bmatrix} .
\end{equation*}

Assuming that $\lambda > 0$ is so that $\lambda\notin\Lambda_2$, all the eigenvalues are simple and, because $\beta_{2,n} \neq 0$ and $\phi_n'(1) \neq 0$, the pairs $(A_0,\mathfrak{B}_0)$ and $(A_0,C_0)$ satisfy the Kalman condition. Assuming now that $\lambda\in\Lambda_2$, there exist eigenvalues of multiplicity 2. Let $1 \leq n < m$ be such that $\sigma_n = \sigma_m$. We check that 
$\begin{vmatrix}
\beta_{1,n} & \beta_{2,n} \\
\beta_{1,m} & \beta_{2,m}
\end{vmatrix}
= 2 n m (m^2-n^2) \pi^4 \neq 0$,
hence the pair $(A_0,\mathfrak{B}_0)$ satisfies the Kalman condition. We also have 
$\begin{vmatrix}
c_{1,n} & c_{1,m} \\
c_{2,n} & c_{2,m}
\end{vmatrix}
= 2 n m \pi^2 \left[ (-1)^n \cos(m\pi\xi) - (-1)^m \cos(n\pi\xi) \right]$, hence if we select the measurement location $\xi\in[0,1)$ such that $(-1)^n \cos(m\pi\xi) \neq (-1)^m \cos(n\pi\xi)$ for all pairs of integers $1 \leq n < m$ satisfying $\lambda = (n^2+m^2)\pi^2$, then the pair $(A_0,C_0)$ satisfies the Kalman condition. This allows the statement of the following theorem, whose proof is analogous to the one of Theorem~\ref{thm1}, and  where we denote by $\tilde{K}_1$ and $\tilde{K}_2$ the two lines of $\tilde{K}$, $E_1$ and $E_2$ the two lines of $E$, and $\mathcal{L}_1$ and $\mathcal{L}_2$ the two columns of $\mathcal{L}$.

\begin{theorem}\label{thm3}
Let $\lambda > 0$ and $\mu \in\R$. Let $\xi\in[0,1)$ be such that $(-1)^n \cos(m\pi\xi) \neq (-1)^m \cos(n\pi\xi)$ for any pair of integers $1 \leq n < m$ satisfying $\lambda = (n^2+m^2)\pi^2$, if any. Let $\delta > 0$ and $N_0 \geq 1$ be such that $\sigma_n < - \delta$ for all $n \geq N_0 + 1$. Let $K\in\R^{2 \times N_0}$ and $L\in\R^{N_0 \times 2}$ be such that $A_0 + \mathfrak{B}_0 K$ and $A_0 - L C_0$ are Hurwitz with eigenvalues that have a real part strictly less than $-\delta<0$. For a given $N \geq N_0 +1$, assume that there exist a symmetric positive definite $P\in\R^{2N \times 2N}$ and positive real numbers $\alpha > 2/\pi^4$ and $\beta_1,\beta_2,\gamma > 0$ such that 
\begin{equation}\label{eq: thm3 constraints}
\Theta_1 \prec 0 , \quad
\Theta_2 < 0 , \quad
\Theta_3 \geq 0
\end{equation}
where
\begin{align*}
\Theta_1 & = \begin{bmatrix}
F^\top P + P F + 2 \delta P + \alpha\gamma \Vert \mathcal{R}_N a_2 \Vert_{L^2}^2 \tilde{K}_2^\top \tilde{K}_2 & P \mathcal{L}_1 & P \mathcal{L}_2  \\
\mathcal{L}_1^\top P & -\beta_1 & 0  \\
\mathcal{L}_2^\top P & 0 & -\beta_2
\end{bmatrix}  \\
& \phantom{=}\; + \alpha\gamma \Vert \mathcal{R}_N b_1 \Vert_{L^2}^2 E_1^\top E_1 + \alpha\gamma \Vert \mathcal{R}_N b_2 \Vert_{L^2}^2 E_2^\top E_2 \\
\Theta_2 & = 2 \gamma \left( \sigma_{N+1} + \delta + \frac{2(N+1)^4}{\alpha} \right) + (\beta_1 + \beta_2) M_\phi \\
\Theta_3 & = (N+1)^2 \left( \pi^4 - \frac{2}{\alpha} \right) - \lambda \pi^2  
\end{align*}
where $M_\phi = 2\pi^2 \sum_{n \geq N+1} \frac{1}{n^2} = 2\pi^2 \left( \frac{\pi^2}{6} - \sum_{n=1}^N \frac{1}{n^2} \right)$. Then, considering the closed-loop system composed of the plant (\ref{eq: RD system 3}) with the system output (\ref{eq: measurement operator 2}) and the controller (\ref{eq: controller 2}), there exist $C,M > 0$ such that for any initial conditions $z_0 \in D(\mathcal{A})$ with $\Vert z_0 \Vert_{H^2} < C$, the system trajectory with null initial condition of the observer ($\hat{z}_n(0) = 0$ for all $1 \leq n \leq N$) satisfies 
\begin{equation}\label{eq: thm3 stability estimate}
\Vert z(t,\cdot)\Vert_{H^2}^2 + \sum_{n=1}^N \hat{z}_n(t)^2 \leq M e^{-2\delta t} \Vert z_0 \Vert_{H^2}^2
\end{equation} 
for all $t \geq 0$. Moreover, the constraints (\ref{eq: thm3 constraints}) are feasible when $N$ is selected large enough.
\end{theorem}

\section{Numerical illustration}\label{sec4}

We consider the case $\lambda = 25$ and $\mu = 1$. Since $\lambda \notin\Lambda$, we can apply the result of Theorem~\ref{thm1}. Considering the feedback gain $K = 2.31$ and the observer gain $L = -34.96$, the constraints (\ref{eq: thm1 constraints}) with $\delta = 0.5$ are found feasible for a dimension of the observer $N=5$. This ensures the local exponential stability of the closed-loop system trajectories evaluated in $H^2$-norm in the sense of (\ref{eq: thm1 stability estimate}).

\section{Conclusion}\label{sec: conclusion}

This paper has addressed the problem of local output feedback stabilization of a nonlinear KSE using a Neumann boundary measurement. Depending on the values of the parameters of the plant, we discussed the choice of the actuation/sensing scheme in order to avoid the pitfall of possible unstable modes that are uncontrollable/unobservable. This led in certain cases to a multi-input multi-output control design procedure. In this context, the proposed control design strategy has been shown to be systematic. It is worth noting that even focused on pointwise Neumann measurement, the approach reported in this paper can also be applied to the less stringent case of pointwise Dirichlet measurement for PDE trajectories that can be evaluated either in $H^1$-norm or $H^2$-norm. To conclude, let us point out that the approach reported in this paper can also be easily used to locally stabilize the classical Burger equation~\cite{krstic1999global}.






\ifCLASSOPTIONcaptionsoff
  \newpage
\fi



\bibliographystyle{IEEEtranS}
\nocite{*}
\bibliography{IEEEabrv,mybibfile}

\end{document}